\documentclass[10pt]{amsart}
\usepackage{amsmath,amssymb, amsthm} %if documentclass is amsart, amsmath package isn't needed
\theoremstyle{plain}
    \newtheorem{thm}{Theorem}
    \newtheorem{cor}[thm]{Corollary}
    
    \newtheorem{prop}{Proposition}
    
    \newtheorem{lemma}[prop]{Lemma}

\theoremstyle{definition}
    
\theoremstyle{remark}
    \newtheorem{rem}[prop]{Remark}

\def\DC{{\mathrm{DC}}}

\def\SU{{\mathrm{SU}}}

\def\be{\begin{equation}}
\def\ee{\end{equation}}

\def\bm{\begin{matrix}}
\def\em{\end{matrix}}

\newcommand{\C}{\mathbb{C}}\newcommand{\R}{\mathbb{R}}\newcommand{\Q}{\mathbb{Q}}\newcommand{\Z}{\mathbb{Z}}

\newcommand{\D}{\mathbb{D}}
\renewcommand{\P}{\mathbb{P}}

\newcommand{\cE}{\mathcal{E}}

\renewcommand{\setminus}{\smallsetminus}

\newcommand{\SL}{\mathrm{SL}}

\newcommand{\id}{\mathit{id}}

\newcommand{\comm}[1]{}

\begin{document}

\title[KAM, Lyapunov exponents, and the Spectral Dichotomy]
{KAM, Lyapunov exponents, and the Spectral Dichotomy
for typical one-frequency Schr\"odinger operators}

\author[A.~Avila]{Artur Avila}
\address{
Institut f\"ur Mathematik, Universit\"at Z\"urich, Winterthurerstrasse 190,
CH-8057 Z\"urich, Switzerland and IMPA, Estrada Dona Castorina, 110,
Rio de Janeiro, 22460-320, Brazil}
\thanks{This work was partially conducted during the period the author
served as a Clay Research Fellow.  It was supported by an ERC Starting Grant
and a grant from the SNSF}
\email{artur.avila@math.uzh.ch}

\begin{abstract}
We show that a one-frequency analytic $\SL(2,\R)$
cocycle with Diophantine rotation vector is analytically linearizable if and
only if the Lyapunov exponent is zero through a complex neighborhood of the
circle.  More generally, we show (without any arithmetic assumptions) that
%under a Diophantine condition solely
%on the base frequency,
regularity implies almost reducibility, i.e., the range of
validity of the perturbative analysis near constants
is specified by a condition on the
Lyapunov exponents.  Together with our
previous work, this establishes a Spectral Dichotomy for
typical one-frequency Schr\"odinger operators: they can be written as a
direct sum of large-like and small-like operators.  In particular, the
typical operator has no singular continuous spectrum.
\end{abstract}

\date{\today}

\maketitle

%\keywords{Symplectic diffeomorphisms, partial hyperbolicity, Lyapunov exponents, generic properties}

%\subjclass[2000]{primary;secondary}
%37D20 Uniformly hyperbolic systems (expanding, Anosov, Axiom A, etc.)
%37D25 Nonuniformly hyperbolic systems (Lyapunov exponents, Pesin theory, etc.)
%37D30 Partially hyperbolic systems and dominated splittings
%37J10 Symplectic mappings, fixed points
%37C20 Generic properties, structural stability

%\date{\today} %here if class=amsart
%\maketitle   %here if class=amsart

\section{Introduction}

This paper has two main purposes.  From the dynamical side we will show that
for analytic $\SL(2,\R)$ cocycles over a translation of $\R/\Z$,
the ``domain of applicability of KAM techniques'' is given precisely by a
condition on the Lyapunov exponent of complexifications, which closes a
chapter opened in \cite {H}.  From the spectral side, it provides the last
step in our program to prove the Spectral Dichotomy for typical
one-frequency Schr\"odinger operators.

\subsection{``...et quelques examples montrant
le caract\`ere local d'un th\'eor\`eme d'Arnold et de Moser
sur le tore de dimension deux''}  This is the second half of the title of
Michael Herman's celebrated paper \cite {H}, where he introduced
his famous ``subharmonic method''
for minoration of the Lyapunov exponents of cocycles, which he then uses to
make explicit the ``lack of
globality'' of KAM theory on the two-dimensional torus (in striking contrast
with the one-dimensional situation by the remarkable Herman-Yoccoz
rigidity theory \cite {H1}, \cite {Y}).

The main class of ``counterexamples to global KAM'' in the two-torus
discussed in \cite {H} is given by one-frequency $\SL(2,\R)$ cocycles
homotopic to a constant.
Those have the form $(\alpha,A):
(x,y) \mapsto (x+\alpha,A(x) \cdot y)$, where $\alpha \in \R \setminus \Q$
and $A(x)$ is a continuous $\SL(2,\R)$ function of $x$, understood to act
projectively in the second coordinate.  In the case
\be
A(x)=A^{(E-\lambda v)}(x)=\left (\bm E-\lambda v(x) & -1 \\ 1 & 0 \em \right
),
\ee
where $\lambda>0$ is some coupling constant and $v$ is a
trigonometric polynomial, the
argument of Herman can be summarized as follows.  For
general reasons, $(\alpha,A)$ has always a well defined rotation vector
$(\alpha,\rho)$ with $\rho \in [0,1]$, and all values of $\rho$ can be
obtained as $E$ changes.  In particular, if $\alpha \in \DC$, i.e.,
it satisfies a Diophantine condition, then one can
always choose $E$ such that
the vector $(\alpha,\rho)$ satisfies some
fixed Diophantine condition.  Thus if $\lambda$
is suitably small, one
can apply KAM theory (this idea in fact dates back to Dinaburg-Sinai \cite
{DS}) to show that the torus dynamics is in fact
quasiperiodic.  On the other hand, he showed that if
$\lambda$ is large then
the cocycle defined by $(\alpha,A)$ has a positive Lyapunov exponent
\be \label {L}
L=\lim \frac {1} {n} \int_{\R/\Z} \ln \|A(x+(n-1) \alpha) \cdots A(x)\|
dx>0,
\ee
which implies that the dynamics on the two-torus
has exactly two ergodic invariant measures with one zero Lyapunov exponent
corresponding to the basis, and one non-zero Lyapunov exponent
corresponding to the fiber, which is either $2L$ or $-2L$ according to the
measure.  In particular the dynamics is not
quasiperiodic, even though the rotation vector can be chosen to satisfy a
Diophantine condition.

Around 1990, two key developments complemented Herman's theory.  On the
large $\lambda$ side,
Sorets-Spencer \cite {SS} showed that a minoration of the
Lyapunov exponent can be proved for real analytic functions (and not only
trigonometric polynomials).
On the small $\lambda$ side,
Eliasson \cite {E} proved that a non-standard KAM algorithm allows
to ``control the
dynamics'' for all values of $E$, introducing the concept of {\it almost
reducibility} in this context.  More recently, non-KAM methods \cite {AJ2}
led to the following result: for any fixed real analytic $v$,
if $\lambda$ is small then for every
$\alpha \in \DC$ and any $E \in \R$, the dynamics is almost reducible,
that is, up to coordinate changes, it can be made arbitrarily close to a
product dynamics $(x,y) \mapsto (x+\alpha,A_* \cdot y)$
(as an analytic dynamical system), that is, $A_*$ does not depend on $x$.

Clearly the powerful notion of almost reducibility captures precisely what
it means to ``belong to the domain of KAM techniques''.  So the natural
follow up to the title of this section is: What are then the obstructions to
almost reducibility?  As Herman pointed out, the behavior of the Lyapunov
exponent is part of the answer.  We will show that it is in fact the whole
answer.

Notice that it is not exactly the positivity of the Lyapunov exponent that
is to blame for the failure of almost reducibility: indeed a cocycle may
very well be conjugate to $(\alpha,A_*)$ with $A_*$ a hyperbolic matrix. 
It is only by looking at the complexification (the main idea introduced in
\cite {H}) that one gets to a key feature of constant cocycles
that {\it must} be shared by almost reducible cocycles, called
{\it regularity}:\footnote {The notion of regularity was introduced in \cite
{A1} for $\SL(2,\C)$ valued cocycles, but coincides with the one given here
when restricted to $\SL(2,\R)$ cocycles.}
The Lyapunov
exponent remains constant under small perturbations in the imaginary
direction.  In other words the value of $L$ in (\ref {L}) must remain
unchanged when the integration is taken over $\R/\Z+\epsilon i$. 

The Almost Reducibility Conjecture (ARC) states precisely that regularity is
equivalent to almost reducibility.\footnote {We note that the ARC is usually
stated just in the case of zero Lyapunov exponents, since
the positive Lyapunov exponent case is known.}  It lies clearly beyond the
scope of local reducibility theory, since it assumes no proximity to a well
understood model.\footnote {While non-local reducibility results
exist, based on renormalization, the required
{\it a priori} bounds currently depend on no average growth of
the cocycle (see \cite {AK1}, \cite {AK2}), while a hypothesis on the
Lyapunov exponent can only give slow growth at most.}  We have previously
established the ARC under an exponentially Liouville hypothesis on the
frequency \cite {arac}.  Here we treat the complementary case when
$\beta=\limsup \frac {\ln q_{n+1}} {q_n}=0$, where $q_n$ are
the denominators of the continued fraction approximants of $\alpha$.

\begin{thm} \label {ar}

Let $\alpha \in \R \setminus \Q$ and
$A \in C^\omega(\R/\Z,\SL(2,\R))$ be such that
$(\alpha,A)$ is regular.  Then $(\alpha,A)$ is almost reducible.

\end{thm}

\begin{cor}

Let $\alpha \in \R \setminus \Q$
and $A \in C^\omega(\R/\Z,\SL(2,\R))$ be such that
$(\alpha,A)$ has a Diophantine rotation vector.  Then $(\alpha,A)$ is
regular if and only if $(\alpha,A)$ is analytically conjugate to a constant
translation.

\end{cor}

\begin{proof}

Almost reducibility reduces the analysis to
the local case which was done in \cite
{AJ2}.
\end{proof}

However nice it is to establish a conceptual connection between
Lyapunov exponents and perturbative theory, one is still justified to wonder
what is to be gained in practice.  The answer lies in the recent advances in
\cite {A1} that introduced the tools to analyze the parameter
dependence of the Lyapunov
exponent of complexifications.  Combined with the
results of this paper, one does get a remarkably precise description of the
phase transitions as one consider the global situation, that is, how one
moves from KAM-like to non-uniformly hyperbolic dynamics, the most striking
consequences of which are seen in the application to
the theory of one-frequency Schr\"odinger operators to which we now turn.

\subsection{Global theory of one-frequency Schr\"odinger operators}

Cocycles of the form $(\alpha,A^{(E-v)})$ arise in the consideration of
one-frequency Schr\"odinger operators with frequency $\alpha$ and analytic
potential $v$, i.e., bounded self-adjoint operators
$H=H_{\alpha,v}$ on $\ell^2(\Z)$ of the form\footnote{It is common to
consider a third parameter, a phase $\theta$, in the definition of $H$. 
This corresponds to considering the shifted potential $v_\theta(\cdot)=
v(\theta+\cdot)$.}
%determined by a {\it frequency} $\alpha \in \R \setminus \Q$ and a {\it
%potential} $v \in C^\omega(\R/\Z,\R)$:
\be \label {sch}
(Hu)_n=u_{n+1}+u_{n-1}+v(n \alpha) u_n.
\ee
While such operators have been an
important topic of research since the 1970's, the understanding was mostly
restricted to two local regimes, corresponding to operators with
large-like and small-like potentials.  Only very recently
an approach towards a global theory has emerged \cite {A1}.

One of the most
remarkable accomplishments of the local theories was a good
understanding of the spectral measures in the two respective
regimes: typically point spectrum for large-like potentials \cite {BG}
and absolutely continuous spectrum for small-like potentials (\cite {E},
\cite {BJ2}, \cite {AJ2}, \cite {arac}).  In \cite {A1} we set as one of the
goals of the global theory to prove that
a typical (in a measure-theoretical sense)
operator $H$ has no singular continuous spectrum.  This is achieved here:

\begin{thm} \label {sing}

For a typical Schr\"odinger operator (\ref {sch}), the spectral measures
have no singular continuous component.

\end{thm}

In fact we get a much more precise description of the spectrum.  Recall that
in \cite {A1}, it is shown that any energy in the spectrum can be
classified into one of three classes, {\it subcritical}, {\it critical} and
{\it supercritical}, so that the critical locus is at most
``codimension-one''.  From the point of view of this classification,
in \cite {A1}, we achieved a best possible
description of the phase transition behavior for a typical quasiperiodic
operator: it is shown that for every $\alpha \in \R
\setminus \Q$ and for a typical $v \in C^\omega(\R/\Z,\R)$, the
corresponding operator is {\it acritical}, that is,
the spectrum $\Sigma=\Sigma_{\alpha,v}$
splits as a disjoint union of compact sets
$\Sigma_-$ and $\Sigma_+$ corresponding to subcritical and supercritical
energies.  Moreover, any small perturbation of $H$ (which may involve both
the frequency and the potential) is acritical as well, and the corresponding
splitting of $\Sigma_{\alpha,v}$
depends continuously on the perturbation.

The definition of criticality in \cite {A1} is based on the
behavior of the Lyapunov exponent of phase complexifications of the cocycle
$(\alpha,A^{(E-v)})$ associated to the Schr\"odinger equation $Hu=Eu$:
for $E$
in the spectrum, subcritical corresponds to regularity, supercritical to
positivity of the Lyapunov exponent, and critical to the complement.
The immediate question such classification
raises is whether the non-critical regimes do behave
according to the local theories, as the terminology suggests.

One of the nice properties of acritical operators is that the Lyapunov
exponent is bounded from below on $\Sigma_+$.
The behavior for energies with bounded from below Lyapunov exponent
was extensively analyzed in the non-perturbative theory of
Bourgain, Goldstein and Schlag, in particular:
\begin{enumerate}
\item Under a full measure arithmetic condition on the frequency, the
integrated density of states (i.d.s.)
is locally H\"older \cite {GS},
\item Up to a typical small perturbation of the frequency, there is
Anderson localization (pure point spectrum with exponentially
decaying eigenfunctions) \cite {BG} and the i.d.s. is absolutely continuous
\cite {GS1}.
\end{enumerate}

The analysis of the behavior at $\Sigma_-$ is much more recent, in fact the
notion which would later be called subcriticality first appears in passing
in \cite {AJ2}, where it is first speculated to be connected to
almost reducibility.  The corresponding region of
the spectrum $\Sigma_{ar}$ (which is open by \cite {arac})
can be analyzed by either KAM theory \cite {E}, \cite {AFK}
or Aubry duality \cite {AJ2}, in particular
\begin{enumerate}
\item The spectral measures are absolutely continuous \cite {AJ2}, \cite
{arac},
\item Under a full measure arithmetic condition on the frequency, the
i.d.s. is locally $1/2$-H\"older
\cite {AJ2}, \cite {amor}.
\end{enumerate}

%Our main achievement in this paper is to establish, for typical frequencies,
%the connection between almost reducibility and subcriticality.

%\begin{thm}

%If $\alpha \in \DC$ then $\Sigma_-=\Sigma_{ar}$.

%\end{thm}

%Put it another way, what we show is that the ``domain of validity of
%KAM'', in the theory of one-frequency Schr\"odinger operators,
%is caracterized by the behavior of the Lyapunov exponent of the
%complexification.

Our main theorem implies that $\Sigma_-=\Sigma_{ar}$.
Combined with the prevalence of acriticality, we obtain the following
Spectral Dichotomy: a typical operator $H$ can be written as a direct sum
of a large-like operator $H_+$ and a small-like operator $H_-$ with disjoint
spectra (just define $H_\pm$ using the spectral projections associated to
$\Sigma_\pm$).

As a consequence of our understanding of $\Sigma_+$ and $\Sigma_{ar}$, we get
Theorem \ref {sing} (with Anderson localization in $\Sigma_+$ and absolutely
continuous spectrum in $\Sigma_-$).  Moreover, we also have:

\begin{thm}

For a typical Schr\"odinger operator (\ref {sch}), the integrated density of
states is H\"older and absolutely continuous.

\end{thm}

\begin{rem}

Absolute continuity of the i.d.s. implies that the spectrum is ``everywhere
thick'' in the sense that any non-empty intersection with an open set must
have positive Lebesgue measure.
Let us note for completeness that it was known that
the spectrum of a typical operator $H$ is
a Cantor set.  This is a direct
consequence of \cite {GS2} (which shows that up to
a typical perturbation of $\alpha$, $\Sigma_+$ has empty interior) and
\cite {density} (which shows that up to a typical perturbation of $v$,
$\Sigma \setminus \Sigma_+$ has empty interior).

\end{rem}

\begin{rem}

As discussed above, before our work the supercritical theory could be
considered to be much more developed than the subcritical one.
This is reverted with the ARC.  Indeed
the supercritical analysis is still by no means complete in crucial places,
expecially since many
results prove typical properties by a parameter exclusion argument on the
frequency.  While some exclusion is necessary (particularly for
localization), it would be desirable for it
to be given in terms of explicit arithmetic conditions (with any further
parameter
exclusion being relegated to the potential, preferably just through a phase
shift).  On the other hand, in the subcritical region we have now a far
better understanding, with many results (such as absolute continuity of
the spectral measures \cite {arac}) not depending at all on parameter
exclusion.

\end{rem}

\begin{rem} \label {developments}

As we hoped, the Global Theory and the ARC have found many applications.  It
is used for instance in joint work with Jiangong You and Qi Zhou
to solve the Dry Ten Martini problem in the
non-critical case \cite {AYZ}.
Another exciting direction of progress is the use of
duality techniques in this context, which gives new information about the
supercritical region as well (addressing some of the outstanding issues
discussed in the previous remark), as presented by Jitomirskaya in
her Plenary Address at the 2022 ICM.

\end{rem}

{\bf Acknowledgements:} The main result of this paper was proved in 2009
(in the Diophantine case) and 2012 (in the more general $\beta=0$ case). 
Descriptions were given in a course in the Fields institute in 2011,
and during a meeting in Oberwolfach in 2012.
I thank the community for their continuous encouragement to
provide a written account.  To present accurately the history, the
introduction and the bibliography reflect my understanding of the
subject in 2012, with the exception of Remark \ref {developments}.
I thank my PhD students Fernando Argentieri,
Pedram Safaee and Andrea Ulliana for the detailed reading of
this text.

\section{What is happening in the proof}

The following is an explanation more than an outline.  Since
the proof is short we recommend the reader to
first read it and then to come back here for the narrative.

The argument is mostly based on complex analysis, and is mostly
self contained.
The preparatory section \ref {preparation}
in particular collects the main outside ingredients
(coming from \cite {A1} and \cite {BJ1}), their use
is pointed out below.

An $\SL(2,\R)$ cocycle naturally gives rise to a one-parameter family
of perturbations
obtained by composition with rigid rotations $R_\theta$.
For imaginary parameters
$i \theta$, such
cocycles are uniformly hyperbolic at real $z$ (by hyperbolic geometry).
Uniformly hyperbolic cocycles can be
diagonalized by straightening out the unstable and stable directions
$u_\theta$ and $s_\theta$.  If
we have good enough bounds on this conjugacy $B$
(as the perturbation becomes smaller), it will also (complex)
conjugate the initial cocycle close to diagonal.  Indeed if
$\|B\|^2=\theta^{-1+\kappa-o(1)}$ then closeness is $\theta^{\kappa-o(1)}$. 
Getting $\kappa$ non-negative is already non-obvious, but we need a positive
$\kappa$ (and there is not much room, since $\kappa$ can not go beyond $1/2$
when starting with a parabolic cocycle).

The first step is to show that the uniform hyperbolicity of the perturbed
matrices does persist for non-real $z$ since the hyperbolic geometry
argument only works at the real line.  This follows from regularity through
a central result of the first part of \cite {A1}.

Given the existence of the extensions, it is possible to show that
the size of the conjugacy $\|B\|^2$ is basically the inverse of
the angle
between the unstable and stable directions.  This is
actually subtle because we need to obtain holomorphic solutions with optimal
behavior regarding some real geometric constraint, see section \ref {mini}.

The estimate on the angle (section \ref {angle})
is our main concern.  The basic idea is that if
the angle is small then the Lyapunov exponent $L_\theta$
is small.  But if the Lyapunov
exponent is small ($L_\theta=\theta^{\kappa(\theta)}$ with $\kappa(\theta)$
large) then the angle must be often
large ($\theta^{1-\kappa(\theta)}$) at real $z$ (this is a
hyperbolic geometry estimate basic to Kotani theory, and shows in particular
that $0<\kappa(\theta)<1$).  These
two estimates work together to deal with non-real $z$ by subharmonicity. 
Notice that the number $1/2$ appears when combining $\kappa(\theta)$ and
$1-\kappa(\theta)$ and this is essentially sharp as discussed before.

Showing the basic estimate relating angle and the Lyapunov exponent is done
in two rounds.  First we need an a priori bound in order to know that
the angle is not terrible.  This depends on a quantification of the
argument of \cite {A1} which establishes uniform hyperbolicity within
the domain of cocycles which are regular and have positive Lyapunov
exponent.  It shows that weak hyperbolicity (small angle)
can only arise near the boundary of this (infinite dimensional)
domain, so we need to control the
distance to the boundary.  Regularity is not an issue, since it is stable,
but the Lyapunov exponent is relatively close to $0$.  Here we use
a continuity estimate on the Lyapunov exponent (due to \cite {BJ1})
to control how fast it could drop to $0$.  In the second round we show
that the angle satisfies the Lyapunov exponent bound
in a large (equidistributed) set,
and in the presence of the a priori bound this allows us to get this
better bound everywhere through a Brownian motion argument.

With the angle estimate done we can complex conjugate the cocycle close to
rotations (section \ref {compsec}).
This provides in particular good polynomial growth estimates for
the initial cocycle.  More crucially, it gives good bounds on
vector representations $U_\theta=B^{-1} \cdot \begin{pmatrix} i\\1
\end{pmatrix}$ and $S_\theta=B^{-1} \cdot \begin{pmatrix} -i\\1
\end{pmatrix}$ of the unstable and stable directions of the
perturbations.  We also derive a better upper bound on $L_\theta$,
$\kappa(\theta)>1/2-o(1)$.

The conjugacy obtained so far was built on the basis of the unstable and
stable directions of complex perturbations of the cocycle.  Since the
complex perturbations are not real symmetric, neither are the conjugacies
$B$.  In order to fix this issue (section \ref {realsec}), we pair
$u_\theta$ with
its symmetrization $u'_\theta$ and try to straighten then out with
a real symmetric $B_r$. The issue is again to get a lower bound on
the angle.  But now it is easier, since we have vector representations
$U_\theta$ and $U'_\theta$ we can use the determinant $\Delta$ to
get information on the angle.  If the angle is small somewhere (at
some $z_*$) this must propagate through an equidistributed set (of
translates of $z_*$) parallel to the real line.  This makes the determinant
drop on this set, and since the determinant is a nice function a Lagrange
interpolation argument shows that it satisfy a similar bound on the
real line.  But on the real line we have the hyperbolic geometry lower
bound for the angle (now improved due to the new upper bound on $L_\theta$),
and too small determinant means that $U_\theta$ must be small.  If this does
not happen, we get the angle estimate, as shows that the angle could not
have been small at any $z_*$.  If $U_\theta$ is indeed small, we just
observe that $S_\theta$ can not be small at the same time (the determinant
of $U_\theta$ and $S_\theta$ having been normalized as $2 i$ by definition),
so we obtain an angle estimate for the pairing of $s_\theta$ and its
symmetrization $s'_\theta$.

\section{Some ingredients} \label {preparation}

\def\d{{\underline{d}}}
\def\H{{\mathbb {H}}}
\def\PSL{{\mathrm{PSL}}}

Below, once $\alpha \in \R$
and $A \in C^\omega(\R/\Z,\SL(2,\C))$ are fixed, we let $A_n(x)=A(x+(n-1)
\alpha) \cdots A(x)$, $L_n=\frac {1} {n} \int \ln \|A_n(x)\| dx$ and $L=\lim
L_n$.

For $\epsilon>0$,
we write $C^\omega_\epsilon(\R/\Z,*)$ for the space of analytic
functions to the target $*$ (typically $\R$, $\SL(2,\R)$,...)
admitting a bounded holomorphic extension to $|\Im
z|<\epsilon$, and we let $\| \cdot \|_\epsilon$ denote the corresponding
$L^\infty$ norm of the extension.  We also let $\|\cdot\|_0$ denote the
$L^\infty$ norm on $\Im z=0$.

We identify $\P \C^2$ with the Riemann sphere $\overline \C$ in the familiar
way $(x,y) \mapsto \frac {x} {y}$.  The action of $\SL(2,\C)$ on $\overline
\C$ is then given by $\begin {pmatrix} a&b\\c&d \end{pmatrix} \cdot z=\frac
{az+b} {cz+d}$, so $\SL(2,\R)$ consists of matrices preserving the upper
half plane $\H$.
We denote $R_\theta=\begin{pmatrix} \cos 2 \pi \theta&-\sin 2 \pi \theta\\
\sin 2 \pi \theta& \cos 2 \pi \theta \end{pmatrix}$.

We let $d(u,s)$ be the absolute value of the
sine of the angle between $u$ and $s$ on
$\P\C^2$, which is invariant under $\SU(2)$.  We note that $-\ln
d$ is plurisubharmonic in the complement of the diagonal in $\P \C^2 \times
\P \C^2$.

Sometimes it is convenient to move from $\H$
to the disk $\D$, so we denote $\dot{z}=\frac {z-i} {z+i}$.

\subsection{Matrix growth}

The next is Lemma 7 of Bourgain-Jitomirskaya \cite {BJ1}.
%(Lemma 7 in \cite {BJ1}).

\begin{lemma} \label {bjestimate}

For every $\xi,\Xi>0$, there exist constants $C>0$, $c>0$, $c'>0$
with the following properties.  Let $A \in C^\omega_\xi(\R/\Z,\SL(2,\C))$ be
such that $\|A\|_\xi<\Xi$.
Let $\alpha$ and $q$ be such that $|\alpha-\frac {p} {q}|<\frac {1} {q^2}$.
Let $0<\kappa<1$, let $N>C \kappa^{-2} q$ and let $N'$ be a multiple of
$N$ with $\frac {N'} {N}<e^{e^{c \kappa q}}$.  If $L_{2 N}>\frac {9} {10}
L_N$ and $L_N>100 \kappa$ then
\be
|L_{N'}+L_N-2 L_{2N}|<e^{-c' \kappa q}+C \frac {N} {N'}.
\ee

\end{lemma}

\begin{rem}

The result is originally stated for (real)
Schr\"odinger cocycles, but the proof
extends without change.  The original statement also assumes $L>100 \kappa$,
but the proof only uses the weaker condition discussed here.

\end{rem}

\begin{lemma} \label {bjcor}

There exists
$M>1$ with the following property.  Let $\alpha \in \R \setminus \Q$ be such
that $\beta=0$.  Then for every
$\xi,\Xi>0$, if $q=q_n$ is sufficiently large and
$A \in C^\omega_\xi(\R/\Z,\SL(2,\C))$ is
such that $\|A\|_\xi \leq \Xi$ and $L=0$ then $L_{q^M}<q^{-1/M}$.

\end{lemma}

\begin{proof}

Fix some small $\sigma>0$.
Assume that $L_{q^M} \geq q^{-1/M}$ with $M$ large.
Then there exists some $q^5<N<q^M$
such that $L_N>q^{-\sigma}$ and $L_{2N}>\frac {9} {10} L_N$
(for instance, $N$ a suitable power of $2$). 
Let $q'=q_{n'}$ be minimal with $q'>e^{q^{1/2}}$,
and let $N'>q'^5$ be the smallest multiple
of $N$.  Then $N'/N=e^{o(e^{q^{1/2}})}<e^{e^{c \kappa q}}/2$, where   
$\kappa=q^{-\sigma}/100$ (here we use $\beta=0$).
It follows from Lemma \ref {bjestimate}
that $L_{N'}$ and $L_{2 N'}$ are
$e^{-c' \kappa q}+C q^M e^{-5 q^{1/2}}<e^{-4 q^{1/2}}$
close to $2 L_{2 N}-L_N$.  This allows us to
iterate the procedure starting with $N'$ instead of $N$.  We conclude that
$L>L_N/2>0$.
\end{proof}

\subsection{Complex perturbations}

The result below is of a quite general nature, holding
for continuous cocycles over a homeomorphism of a compact
metric space.

%For $x \in \overline \C$, let $\d(x)=(x-i)/(x+i)$.  Thus $\d(x)
%\in \D$ if and only if $x \in \H$.

\begin{lemma} \label {field}

Let $\alpha \in \R \setminus \Q$ and $A \in
C^\omega(\R/\Z,\SL(2,\R))$.
%\footnote {The result holds in much more generality, for instance
%in the setting of continuous $\SL(2,\R)$ cocycles
%over homeomorphisms of a compact metric space.}
Then for every $\theta>0$, $(\alpha,R_{-i \theta} A)$ is uniformly
hyperbolic, and the unstable direction satisfies $u(x) \in \H$ and
$|\dot{u}(x)|<e^{-4 \pi \theta}$.
Moreover
\be
L(\alpha,R_{-i \theta} A)=2 \pi \theta+\frac {1} {2} \int_{\R/\Z} \ln
\frac {1-|\dot {u}(x)|^2}
{1-e^{8 \pi \theta} |\dot{u}(x)|^2} dx \geq 2 \pi \theta.
\ee

\end{lemma}

\begin{proof}

It is obvious to check that the action of $R_{-i \theta} A$ takes
$\H$ into the precompact subset
$\H_\theta=\{|\dot{x}|<e^{-4 \pi \theta}\} \subset \H$, thus
uniform hyperbolicity follows by the usual cone field criterium.
The Lyapunov exponent can then be computed as $-\frac {1} {2}$
the average rate of projective contraction of the inclusion
$\H_\theta \to \H$ at the unstable direction,
calculated with respect to any conformal metric: the given formula
corresponds to the choice of the Poincar\'e metric on $\H$.
\end{proof}

\subsection{Robustness of uniform hyperbolicity and angle estimate}

The next result
is a quantification of a basic estimate in the proof of Theorem 6
of \cite {A1}.

\begin{thm} \label {harmo}

Let $\alpha \in \R \setminus \Q$ and let $A \in
C^\omega_{\epsilon_0}(\R/\Z,\SL(2,\C))$ be such that
$(\alpha,A(x+t i))$ is uniformly hyperbolic for every
$|t|<\epsilon_0$ and let $u$ and $s$ be the unstable and stable
directions.  Let
$0<\epsilon_1<\epsilon_0$, $0<\epsilon_2<\min
\{\epsilon_1,\epsilon_0-\epsilon_1\}$ and
$0<\delta<1$ be such that for both $t=\epsilon_1$ and $t=-\epsilon_1$,
for every $\tilde A \in C^\omega_{\epsilon_2}(\R/\Z,\SL(2,\C))$
such that  $\|\tilde A(x)-A(x+t i)\|_{\epsilon_2}<\delta$ the cocycle
$(\alpha,\tilde A)$ is uniformly hyperbolic.  Then $\inf_{x \in \R/\Z}
d(u(x),s(x)) \geq C^{-1} \delta$, where
$C=C(\epsilon_0,\epsilon_1,\epsilon_2,\|A\|_{\epsilon_0})$.

\end{thm}

\begin{proof}

Given two distinct
directions
$\mu,\nu \in \P\C^2$,
let $\begin{pmatrix} a&b\\c&d \end{pmatrix} \in \SL(2,\C)$
be any matrix whose columns are aligned with $\mu$ and $\nu$.  Then
$Q_2=Q_2(\mu,\nu)=-ab$ and $Q_3=Q_3(\mu,\nu)=cd$ are well defined and depend
holomorphically on $\mu$ and $\nu$.
We notice that $\eta \leq d(\mu,\nu)^{-1} \leq \sqrt{5} \eta$
where $\eta=\max \{1,|Q_2|,|Q_3|\}$.\footnote {Obviously $d(\mu,\nu)^{-2}
\geq 1$, and since
$d(\mu,\nu)^{-2}=(|a|^2+|c^2|) (|b|^2+|d|^2)$, we have the estimate
%(the square of the sine of the angle between the columns of a matrix of
%determinant $1$ is the product of the square of the norms of the columns).
$Q_2^2+Q_3^2 \leq
(|a|^2+|c^2|) (|b|^2+|d|^2) \leq (|ab|+|cd|)^2+(|ad|-|bc|)^2
\leq (Q_2+Q_3)^2+1
\leq 5 \eta^2$.
}

%\footnote {For a matrix of determinant $1$, the inverse of the angle
%between the columns is of order the product $\omega$ of their norms.
%Note that $\omega \geq \eta$.  If $|a| \geq |c|/2$ and $|b| \geq |d|/2$
%then $|Q_2| \geq \omega/5$.  If $|c| \geq |a|/2$ and $|d| \geq |b|/2$
%then $|Q_3| \geq \omega/5$.  If $|a| \geq 2 |c|$ and $|d| \geq 2 |b|$ then
%$1=ad-bc \geq 3 |ad|/4$ and $\omega \leq 5 |ad|/4 \leq 5/3$.  If
%$|c| \geq 2 |a|$ and $|b| \geq 2 |d|$ then $1=ad-bc \geq 3 |bc|/4$ and
%$\omega \leq 5 |bc|/4 \leq 5/3$.}

For $j=2,3$, let $Q_j(z)=Q_j(u(z),s(z))$ be defined on $|\Im z|<\epsilon_0$.
We need to show that $\|Q_j\|_0 \leq C \delta^{-1}$, $j=2,3$.

For $w \in C^\omega(\R/\Z,\C)$ let
$\cE^w_2=\begin{pmatrix} 1&w\\0&1 \end{pmatrix}, \cE^w_3=
\begin{pmatrix} 1&0\\w&1 \end{pmatrix} \in C^\omega(\R/\Z,\SL(2,\C))$.

By Lemma 9 of \cite {A1}, for any $|t|<\epsilon_0$, and for $j=2,3$,
the derivative of $\lambda
\mapsto L(\alpha,A(x+t i) \cE^{\lambda w}_j)$ at $\lambda=0$ is given by
$\Re \int_{\R/\Z} Q_j(x+t i) w(x) dx$.
The hypothesis implies that for $t=\pm \epsilon_1$,
the Lyapunov exponent is a
bounded pluriharmonic function on a ball of radius $\delta$ around $A(x+t
i)$ in $C^\omega_{\epsilon_2}(\R/\Z,\SL(2,\C))$, implying
\be
\left |\Re \int_{\R/\Z} Q_j(x \pm \epsilon_1 i) w(x) dx\right|
\leq C' \delta^{-1}
\|w\|_{\epsilon_2},
\ee
for $j=2,3$.  By considering monomials $w=\gamma e^{-2 \pi i k x}$ with
$k \in \Z$ and $|\gamma|=1$,
this shows that the $k$-th Fourier coefficient of $Q_j$ must be bounded by
$C' \delta^{-1} e^{(\epsilon_2-\epsilon_1)|k|}$ as desired.
\end{proof}

\section{Angle estimate} \label {angle}

Through this section, we fix $\alpha \in \R \setminus \Q$ and $A \in
C^\omega_{\epsilon_0}(\R/\Z,\SL(2,\R))$ such that $L(\alpha,A(x+t i))=0$ for
$|t|<\epsilon_0$.
By stability of regularity \cite {A1}, for every
$0<\epsilon<\epsilon_0$, there exists $\theta(\epsilon)>0$
such that for every $0<\theta<\theta(\epsilon)$,
$L(\alpha,R_{-i \theta} A(x+t i))$ is a constant $L_\theta$
for $|t|<\epsilon$.
Since $L_\theta \geq 2 \pi \theta$, by Lemma \ref {field}, we can
write $L_\theta=\theta^{\kappa(\theta)}$ with $\kappa(\theta) \leq 1$.
%For $\theta>0$,
%$(\alpha,R_{-i \theta} A)$ is uniformly hyperbolic by Lemma
%\ref {field}.
Since regularity with positive Lyapunov exponent implies uniform
hyperbolicity \cite {A1}, $(\alpha,R_{-i \theta} A(x+t i))$ is
uniformly hyperbolic for $0<\theta<\theta(\epsilon)$ and $|t|<\epsilon$.
For $0<\theta<\theta(\epsilon)$, let
$u_\theta,s_\theta$ be the unstable and stable directions, which are
holomorphic functions on $|\Im x|<\epsilon$ with values in $\P \C^2$.

The proof of the following result will take the remaining of this section:

\begin{thm} \label {ang}

Assume that $\beta=0$, let $0<\epsilon<\epsilon_0$ and let
$0<\kappa<\min \{\frac {1} {2},1-\frac {\epsilon} {\epsilon_0}\}$.  Then
for every $\theta>0$ sufficiently small we have
\be
d(u_\theta(x),s_\theta(x)) \geq \theta^{1-\kappa}, \quad |\Im x|<\epsilon.
\ee
Moreover we also have
\be
d(u_\theta(x),s_\theta(x)) \geq \theta^{\kappa(\theta)+o(1)},
\quad |\Im x|<\epsilon_0-o(1).
\ee
   
\end{thm}

%The proof of Theorem \ref {ang} will take the remaining of this section. 
Let $\rho_\theta(z)=-\ln d(u_\theta(z),s_\theta(z))$.
Let $M$ be as in Lemma \ref {bjcor}.

\begin{lemma} \label {start}

If $\beta=0$, for every $0<\epsilon<\epsilon_0$, for every
$\theta>0$ sufficiently small, if $q=q_n$ is such that
$\theta>q^{-1/M}$ then $\rho_\theta(z)<q^M$ through $|\Im
z|<\epsilon$.

\end{lemma}

\begin{proof}

Note that $\frac {1} {q} \ln \|(R_{-i \theta} A)_q\|$ is small if $\theta$ is
small.
If $\tilde A$ is $e^{-q^M/2}$ close to $R_{-i \theta} A$ at $\{\Im z=t\}$,
then as in \cite {AK1},
\be
\|\tilde A_{q^M}^{-1} (R_{-i \theta} A)_{q^M}-\id\|=o(1), \quad \Im z=t,
\ee
uniformly on $|t|<\epsilon'$ for any given $\epsilon'<\epsilon_0$.
This gives
\begin{align}
\frac {1} {q^M} \int \ln \|\tilde A_{q^M}(x+t i)\| dx &\geq
\frac {1} {q^M}
\int \ln \|(R_{-i \theta} A)_{q^M}(x+it)\| dx-o(q^{-M})\\
\nonumber
&\geq
L_\theta-o(\theta) \geq 2 \pi \theta-o(\theta)>q^{-1/M}.
\end{align}
By Lemma \ref {bjcor}, if $\tilde A$ has a well behaved analytic extension
around $\{\Im z=t\}$, then
the Lyapunov exponent for $(\alpha,\tilde A(x+t i))$ is positive.  By
stability of regularity \cite {A1},
$(\alpha,\tilde A(x+t i))$ is regular as well, so it
is uniformly hyperbolic.  By Theorem \ref {harmo}, this gives the estimate
$\rho_\theta<\frac {q^M} {2}+C$, and the result follows.
\end{proof}

\begin{lemma}[Brownian motion argument]

Assume that $\rho_\theta$ is defined and bounded
through $|\Im z|<\epsilon$.  Then for
$0<\epsilon'<\epsilon$ and each $q$ we have
\be
\sup_{|\Im z|<\epsilon'} \rho_\theta \leq
C+e^{-C_0 (\epsilon-\epsilon') q} \sup_{|\Im z|<\epsilon}
\rho_\theta-
\ln L_\theta+\sup_{0 \leq n \leq q-1} 2 \ln \|(R_{-i \theta} A)_n\|_\epsilon,
\ee
where $C_0$ is an absolute constant and $C$ depends only on $\sup_{|\Im
z|<\epsilon} \|R_{-i \theta} A(z)\|$.

\end{lemma}

\begin{proof}

The function $\rho_\theta(x)$ is bounded and subharmonic on $|\Im
x|<\epsilon$.  Let $K \subset \{|\Im z|<\epsilon\}$ be the set of all $z$
such that $\rho_\theta(z)=\min_{\Im w=\Im z} \rho_\theta(w)$.

Note that for each $|t|<\epsilon$, $2 L_\theta$ is the integral over $|\Im
z|=t$ of
$\ln \frac {\|R_{-i \theta} A(z) \cdot u_\theta(z)\|} {\|R_{-i \theta}
A(z) \cdot s_\theta(z)\|}$,
where we abuse notation denoting $u_\theta$ and $s_\theta$ unit vectors in
the corresponding directions.  But
$\frac {\|R_{-i \theta} A \cdot u_\theta\|} {\|R_{-i \theta}
A \cdot s_\theta\|} \leq 1+d(u_\theta,s_\theta)
\frac {\|R_{-i \theta} A\|} {\|R_{-i \theta} A
\cdot s_\theta\|}$.  It follows that
\be
2 L_\theta \leq C' \int d(u_\theta(z+x),s_\theta(z+x)) dx.
\ee
In particular, for $z \in K$ we have $d(u_\theta(z),s_\theta(z)) \geq 2
L_\theta/C'$, which gives
$\rho_\theta(z) \leq -\ln L_\theta+\ln \frac {C'} {2}$.   

Notice that since $(R_{-i \theta} A)_k(z)$ takes $(u_\theta(z),s_\theta(z))$
to $(u_\theta(z+k \alpha),s_\theta(z+k \alpha))$, the corresponding angles 
can only differ by a factor of $\|(R_{-i \theta} A)_k(z)\|^2$ (up to
absolute constant).  Thus for
$z \in \tilde K=\bigcup_{0 \leq k \leq
q-1} K+k \alpha$ we have the estimate
\be
\rho_\theta(z) \leq -\ln L_\theta+\sup_{0 \leq n \leq q-1}
2 \ln \|(R_{-i \theta} A)_n\|_\epsilon+C.
\ee
   
Let us now consider a point $z_0$ with $|\Im z_0|<\epsilon'$,
and let us run the
Brownian motion starting at $z_0$
up to the moment it either touches $\tilde K$ or $\{|\Im
z|=\epsilon\}$.  Since $\rho_\theta$ is subharmonic,
$\rho_\theta(z_0)$ is bounded by the
expectation of its value at the endpoint of the Brownian motion.  Thus
\be
\rho_\theta(z_0) \leq p \sup_{|\Im z|<\epsilon} \rho_\theta(z)+(1-p)
\sup_{z \in \tilde K} \rho_\theta(z),
\ee
where $p$ is the probability that the
Brownian motion reaches $|\Im z|=\epsilon$
without hitting $\tilde K$.

Thus it is enough to show that $p<e^{-C_0 (\epsilon-\epsilon') q}$.  In
order to see this, it is enough to show that for some absolute
constant $C_2>0$, the probability that the Brownian motion started at a
point $z$ with $|\Im z|<\epsilon-2 q^{-1}$ reaches the boundary of the 
square $Q(z)$ centered on $z$ and of side $4 q^{-1}$
without touching $\tilde K$ is
at most $e^{-C_2}$.  But since $q$ is an approximant of $\alpha$, $\{k
\alpha\}_{0 \leq k \leq q-1}$ is $2 q^{-1}$ dense on the circle.
So $\tilde K \cap Q$ intersects each horizontal segment
$J_s=\{z+t+i s\}_{t \in [-q^{-1},q^{-1}]}$ for $s \in [-2 q^{-1},2 q^{-1}]$.
Such a set is clearly a definite obstacle for the Brownian motion (after
rescaling by $q$, it has definite logarithmic capacity) and is hit
with positive probability.
\end{proof}

\begin{lemma} \label {brown1}

If $\beta=0$, then $\rho_\theta \leq -\ln L_\theta+o(-\ln \theta)$
over $|\Im z|<\epsilon$.

\end{lemma}

\begin{proof}

Fix $R>0$ large.  Below the $q_{(j)}$ are specific choices of
continued fraction denominators.
%Choose consecutive $q_0$ and $q_1$ such that $e^{-R
%q_0}<\theta \leq e^{-R q_1}$.

Let $q_{(0)}$ be least with $\theta>q_{(0)}^{-1/2 M}$.
By Lemma \ref {start},
$\rho_\theta \leq q_{(0)}^M$ through a band bigger than $\epsilon$.
Take $q_{(1)}$ minimal with $R \ln q_{(0)} \leq q_{(1)}<q_{(0)}$
(use $\beta=0$).
Apply the Brownian motion argument with $q_{(1)}$.
We get an improved estimate
\be
\rho_\theta \leq e^{-c q_{(1)}} q_{(0)}^M-\ln L_\theta+c_\theta
q_{(1)} \leq -\ln
L_\theta+c_\theta q_{(1)}+o(1)
\ee
where $c_\theta \to 0$ as $\theta \to 0$.
If $q_{(1)}=O(-\ln \theta)$, we are done.

Otherwise, repeat the procedure taking
$R \ln q_{(1)} \leq q_{(2)}<q_{(1)}$ minimal.  We get a new estimate
$\rho_\theta
\leq -\ln L_\theta+c'_\theta q_{(2)}$.  If $q_{(2)}=
O(-\ln \theta)$, we are done.

Otherwise, repeat the procedure taking $R \ln q_{(2)} \leq q_{(3)}<q_{(2)}$
minimal.
We get a new estimate
$\rho_\theta \leq -\ln L_\theta+c''_\theta q_{(3)}$.
Notice that $q_{(3)}<R \ln q_{(1)}$ and
$\theta \leq q_{(1)}^{-1/2 M}$ so $q_{(3)}=O(-\ln \theta)$, so we are done.
\end{proof}

So far we have seen that a smaller $\kappa(\theta)$ (so larger $L_\theta$)
yields better angle estimates. 
However a large $\kappa(\theta)$ turns out to help at $\Im z=0$:

\begin{lemma} \label {u}

For a set of probability $1-o(1)$ in
$\Im z=0$, $|\dot{u}_\theta(z)|<1-\theta^{1-\kappa(\theta)+o(1)}$.

\end{lemma}

\begin{proof}

Direct application of Lemma \ref {field}.
\end{proof}

\noindent {\it Proof of Theorem \ref {ang}.}
Fix $\epsilon<\epsilon'<\epsilon_0$ close to $\epsilon_0$.
By Lemma \ref {u}, the set of $x \in \R/\Z$ such that
$\rho_\theta(x)<(-1+\kappa(\theta)-o(1)) \ln \theta$
has probability at least
$1-o(1)$.  By Lemma \ref {brown1}, we also have
$\rho_\theta(x)<(-\kappa(\theta)-o(1)) \ln \theta$ over any $|\Im
x|<\epsilon'$,
which implies the second claim and also the first claim if
$\kappa(\theta)<1/2$.  Otherwise,
since $\rho_\theta$ is subharmonic, it follows that
\be
\sup_{|\Im x|<\epsilon} \rho_\theta(x) \leq -\ln \theta
\left (\frac {\epsilon} {\epsilon'} (\kappa(\theta)+o(1))+
\left (1-\frac {\epsilon}
{\epsilon'} \right ) (1-\kappa(\theta)+o(1)) \right ),
\ee
which implies the first claim in the case
$\kappa(\theta) \geq 1/2$.
\qed

\section{Choosing a third point} \label {mini}

We call $B \in \SL(2,\C)$ a minimizer (or minimizing) if
\be
\left \|B^{-1} \cdot \left ( \bm i \\ 1 \em \right ) \right \|=
\left \|B^{-1} \cdot \left ( \bm -i \\ 1 \em \right ) \right \|.
\ee
Notice that $B$ is minimizing
if and only if $B^{-1} \cdot \overline \R$
is the great circle $H_{B^{-1} \cdot i,B^{-1} \cdot -i}$ of points $z$
which are at the same distance from both $B^{-1} \cdot i$ and $B^{-1} \cdot
-i$.

Let $x,y \in \overline \C$ be two distinct points.  We let $K_{x,y} \subset
\SL(2,\C)$ be the set of all $B$ such that $B \cdot x=i$ and $B \cdot y=-i$.
We let $k(x,y)=\inf_{B \in K_{x,y}} \|B\|^2$.  It is easy to see
that the infimum
is attained, and this happens precisely
at the set of minimizing elements of $K_{x,y}$.

Note that $k(x,y)$ is a
decreasing function of $d(x,y)$ and so it satisfies the maximal principle.
%It is easy to see that
%if $d(x,y)$ denotes the angle between $x$ and $y$ in the usual
%identification $\overline \C=\P\C^2$,
%$1 \leq k(x,y) d(x,y) \leq 2$.
Indeed, up to an $\SU(2)$ change of coordinates
we may assume that $x,y$ are such that $x=-y=\epsilon i$ with $0<\epsilon
\leq 1$.  Then $d(x,y)=\frac {2 \epsilon} {1+\epsilon^2}$
and a minimizing matrix
$B \in K_{x,y}$ is of the form $R_\lambda D_{\epsilon^{-1/2}}$, where
$\lambda \in \R$ and
$D_\mu$ is the diagonal matrix $\begin{pmatrix} \mu&0 \\0&
\mu^{-1} \end{pmatrix}$,
so that $k(x,y)=\epsilon^{-1}$.  In particular, $1 \leq
d(x,y) k(x,y) \leq 2$.

Notice that at a minimizing matrix $B \in K_{x,y}$ we have
\be
\left \|B^{-1} \cdot \left ( \bm \pm i \\ 1 \em \right )
\right \|^2=k(x,y)+k(x,y)^{-1}.
\ee

\begin{lemma} \label {minim}

Let $u,s:\D \to \overline \C$ be holomorphic functions such that
$k_\D(u,s)=\sup_{z
\in \D} k(u(z),s(z))<\infty$.  Then there exists $B:\D \to \SL(2,\C)$
holomorphic with $B(z) \in K_{u(z),s(z)}$,
$k_\D(u,s)=\|B\|^2_\D=\sup_{z \in \D} \|B(z)\|^2$, and such that
the non-tangential limits $B(z_0)$, $z_0 \in \partial \D$ are
minimizers.

\end{lemma}

\begin{proof}

Let $U \in \SU(2)$ be such that $U \cdot \infty=i$ and $U \cdot 0=-i$.
For $t<1$, let $u_t(z)=u(t z)$, $s_t(z)=s(t z)$.  Choose $B_t(z) \in
K_{u_t(z),s_t(z)}$ depending holomorphically on $z$ in a neighborhood of
$\overline \D$.\footnote {Write $u_t=a_t/c_t$ as a quotient of
holomorphic functions without common zeros, and similarly write
$s_t=b_t/d_t$.  Then one can define $B_t$ so that
$B_t^{-1} U=\frac {1} {(a_t d_t-b_t c_t)^{1/2}}
\left (\bm a_t & b_t\\c_t & d_t \em \right )$.}
Notice that $U^{-1} B_t(z) H_{u_t(z),s_t(z)}$ is a circle centered on $0$.
Let $r_t(z)$ be
its radius.  Using the Hilbert transform, let $\nu_t(z):\D \to \C$ be a
holomorphic function on $\D$, smooth up to the boundary, such that $\Re
\nu_t(z)=-\ln r_t(z)$ for $z \in \partial \D$.  Then $B'_t=
U D_{e^{\nu_t/2}} U^{-1} B_t$ is such that
$B'_t(z) \in K_{u_t(z),s_t(z)}$ for
$z \in \D$
and $B'_t(z)$ takes $H_{u_t(z),s_t(z)}$ to $\overline \R$ for
$z \in \partial \D$.  We conclude that
\be
\|B'_t(z)\|^2=k(u_t(z),s_t(z)) \quad \mathrm {for} \quad
z \in \partial \D,
\ee
so that $\sup_{z \in \partial \D} \|B'_t(z)\|^2 \leq k_\D(u,s)$.
Taking a limit of $B'_t$ as $t \to 1$, we get some $B:\D \to \SL(2,\C)$
holomorphic such that $\|B\|_\D^2 \leq k_\D(u,s)$.
The existence of the non-tangential limits of $B$ implies the existence
of the non-tangential limits of $u$ and $s$.  Then we get
\be
\int_0^1 \|B(e^{2 \pi i x})\|^2 dx \leq \limsup_{t \to 1}
\int_0^1 \|B'_t(e^{2 \pi i x})\|^2 dx=
\int_0^1
k(u(e^{2 \pi i x}),s(e^{2 \pi i x})) dx.
\ee
So the minimizing property $\|B(z)\|^2 \geq
k(u(z),s(z))$ gives $\|B(z)\|^2=k(u(z),s(z))$ for almost
every $z \in \partial \D$, that is, $B(z)$ is a minimizer for almost every
$z \in
\partial \D$.
\end{proof}

\begin{lemma}

Let $u$ and $s$ be as in Lemma \ref {minim}, and let $B,\tilde B:\D \to
\SL(2,\C)$ be holomorphic with $B(z),\tilde B(z) \in
K_{u(z),s(z)}$, $k_\D(u,s)=\|B\|^2_\D=\|\tilde B\|_\D^2$ and the
non-tangential limits of $B$ and $\tilde B$ are minimizers.  Then $\tilde
B=R_\lambda B$ for some $\lambda \in \R$.

\end{lemma}

\begin{proof}

%Since $A$ and $\tilde A$ are bounded, the non-tangential limits
%$A(z)$ and $\tilde A(z)$ exist for almost every $z \in \partial \D$.

Since $B(z),\tilde B(z) \in K_{u(z),s(z)}$, $\tilde B(z)=R_{\theta(z)} B(z)$
for some holomorphic function $\lambda:\D \to \C$.
Since $R_{\theta(z)}$ is bounded, $\Im \lambda(z)$ is bounded as well, so
the non-tangential limits $\lambda(z)$ exist for almost every $z \in
\partial \D$.
For $z \in \partial \D$ we must have $\lambda(z) \in \R$.  Since $\Im
\lambda$ is a bounded harmonic function in $\D$ which vanishes almost
everywhere in $\partial \D$, we conclude that $\Im \theta$ vanishes over
$\D$, so $\lambda(z)$ is a real constant.
\end{proof}

\begin{lemma} \label {cho}

Let $u,s:\R/\Z \to \overline \C$ have holomorphic extensions to $|\Im
z|<\delta$ with
$k_\delta(u,s)=\sup_{|\Im z|<\delta} k(u(z),s(z))<\infty$.
Then there exists $B \in C^\omega_\delta(\R/\Z,\SL(2,\C))$
such that $B(z) \in K_{u(z),s(z)}$ for $|\Im z|<\delta$ and
$\|B\|_\delta^2 \leq C_\delta k_\delta(u,s)$. 
Moreover, for almost every $z_0$ with $|\Im z_0|=\delta$ we have
\be \label {exp}
C_\delta^{-1} k(u(z_0),s(z_0))^{1/2} \leq \left \|B(z_0)^{-1} \cdot
\left ( \bm \pm i \\ 1 \em \right ) \right \| \leq C_\delta
k(u(z_0),s(z_0))^{1/2},
\ee
where $u,s,B$ are defined at $\{|\Im z_0|=\delta\}$
through non-tangential limits. 
Moreover, if $u(z) \in \H$ for $\Im z=0$ and
$s(z)=\overline {u(\overline z)}$ then $B$ can be taken
real-symmetric.

\end{lemma}

\begin{proof}

Let $B'$ be a holomorphic function with $B'(z) \in
K_{u(z),s(z)}$ defined on the strip
$\{|\Im z|<\delta\} \subset \C$,
which is given by Lemma \ref {minim} after
changing coordinates from the disk to the strip.  It satisfies
$\|B'\|_\delta=k_\delta(u,s)$.  Then $z \mapsto B'(z+1)$ has the same
properties, so by the previous lemma we have $B'(z+1)=R_\lambda B'(z)$ for
some $\lambda \in [-1/2,1/2)$.  Then $B(z)=R_{-\lambda z} B'(z)$ is
$1$-periodic holomorphic and satisfies $B(z) \in K_{u(z),s(z)}$ and
$\|B\|_\delta^2 \leq C_\delta k_\delta(u,s)$.

Since $B'(z_0)$ is minimizing at the boundary, (\ref {exp}) follows.

For the last statement, note that (still by the previous lemma)
$\overline {B'(\overline z)}=R_{\lambda'} B'(z)$ for some $\lambda' \in
[-1/2,1/2)$.  Moreover, if $\Im z=0$ we have $B'(z)=R_{\lambda''(z)}
B''(z)$, where $B''(z)$ is any $\SL(2,\R)$ matrix sending $u(z)$ to $i$
and $s(z)$ to $-i$. 
Then for $\Im z=0$ we have
$R_{\overline \lambda''(z)}=R_{\lambda'} R_{\lambda''(z)}$,
so $\Im \lambda''(z)=0$ and $\lambda'=0$, that is $B'$ is
real-symmetric, and the resulting $B$ will be real-symmetric as well. 
\end{proof}

\section{Complex conjugacies} \label {compsec}

\begin{thm} \label {comp}

Let $\alpha \in \R \setminus \Q$ be such $\beta=0$ and let
$A \in C^\omega_{\epsilon_0}(\R/\Z,\SL(2,\R))$ be such that
$L(\alpha,A(x+t i))=0$ for $|t|<\epsilon_0$.  Let $0<\epsilon<\epsilon_0$
and let $\kappa_0=\min
\{\frac {1} {2},1-\frac {\epsilon} {\epsilon_0}\}$.
Then for every $\theta>0$
sufficiently small, there exists $B \in C^\omega_\epsilon(\R/\Z,\SL(2,\C))$
such that $\|B\|_\epsilon^2 \leq \theta^{-1+\kappa_0-o(1)}$ and
$B(x+\alpha) R_{-i \theta} A(x) B(x)^{-1}=R_\lambda$ for some $\lambda \in
\C$ with $\Im \lambda=-\frac {\theta^{\kappa(\theta)}} {2 \pi}$.  Moreover,
$\kappa(\theta)>\frac {1} {2}-o(1)$, so
$\tilde A(x)=B(x+\alpha) A(x) B(x)^{-1}$ satisfies $\|\tilde
A-R_{\Re \lambda}\|_\epsilon \leq \theta^{\kappa_0-o(1)}$.

\end{thm}

\begin{proof}

Fix $0<\kappa<\kappa'<\kappa_0$ and $0<\epsilon<\epsilon'<\epsilon_0$
such that $\kappa'<1-\frac {\epsilon'} {\epsilon_0}$.
For $\theta>0$ small, $(\alpha,R_{-i \theta} A(x+t i))$ is uniformly
hyperbolic for $|t|<\epsilon'$ and by the first claim of Theorem \ref {ang}
the unstable and stable directions
$u_\theta,s_\theta$ form an angle at least $\theta^{1-\kappa'}$.  By Lemma
\ref {cho}, there exists $B' \in C^\omega_{\epsilon'}(\R/\Z,\SL(2,\C))$
with $\|B'\|_{\epsilon'}^2 \leq C \theta^{-1+\kappa'}$, such that
$B' \cdot
u_\theta=i$ and $B' \cdot s_\theta=-i$.  Let $A'(x)=B'(x+\alpha) R_{-i
\theta} A(x) B'(x)^{-1}$.

Since the angle between $u_\theta$ and
$s_\theta$ at points $x$ and $x+\alpha$ is of the same order,
(\ref {exp}) yields $\|A'(x)\| \leq C$
at $|\Im x|=\epsilon'$ (take the vectors $\begin{pmatrix} \pm i\\1
\end{pmatrix}$ and apply $A'(x)$).
Since $A' \cdot i=i$ and $A' \cdot -i=-i$ we have that
$A'(x)=R_{\psi(x)}$ for some function $\psi:\R/\Z \to \C/\Z$ with a
holomorphic extension to $|\Im x|<\epsilon'$ such that $\Im \psi$ is bounded.
Thus for some $k \in \Z$,
$\psi(x)=k x+\phi(x)$ for some $\phi:\R/\Z \to \C$ which admits a
bounded holomorphic extension to $|\Im x|<\epsilon''$, where
$\epsilon<\epsilon''<\epsilon'$.  Since $L(\alpha,R_{-i \theta}
A(x+t i))=L(\alpha,A'(x+t i))$ is constant
$L_\theta=\theta^{\kappa(\theta)}$, it follows that
$k=0$ and $\Im \lambda=-\frac {\theta^{\kappa(\theta)}} {2 \pi}$, where
$\lambda$ is the average of $\phi$.
Using that $\beta=0$ we can solve the cohomological equation
$\phi(x)=w(x+\alpha)-w(x)+\lambda$ with
$\|w\|_\epsilon=O(1)$.  Let $B(x)=R_{-w x} B'(x)$.  Then $B(x+\alpha) R_{-i
\theta} A(x) B(x)^{-1}=R_\lambda$ with
$\|B\|_\epsilon^2=o(\theta^{-1+\kappa})$.
%Then $\tilde A(x)=B(x+\alpha) A(x)
%B(x)^{-1}$ satisfies $\|\tilde A-R_{\hat
%\phi_0}\|_\epsilon=o(\theta^\kappa)$.

%If $\|R_{\hat \phi_0}-R_{\Re
%\hat \phi_0}\|<\theta^\kappa/2$ then
%$\|\tilde A-R_{\Re \hat \phi_0}\|_\epsilon<\theta^\kappa$.

%If $\|R_{\hat \phi_0}-R_{\Re
%\hat \phi_0}\| \leq \theta^\kappa/2$ providing the estimate with $A_*=R_{\Re
%\hat \phi_0}$.  Otherwise $\kappa(\theta)$ is either smaller
%than $\kappa$ or close to $\kappa$, and in particular it is smaller than
%$1/2$ and away from $1/2$.

It remains to show that $\kappa(\theta)>\frac {1} {2}-o(1)$.  Assume that
$\kappa(\theta)<\frac {1} {2}-\delta$.  Then we can argue using the
second claim of Theorem \ref
{ang} to produce an improved $B$ with the bound $\|B\|_\epsilon^2
\leq \theta^{-\kappa(\theta)-o(1)}$ conjugating $R_{-i \theta} A$ to
some $R_\lambda$ with $\Im \lambda=-\frac {\theta^{\kappa(\theta)}} {2 \pi}$.
Then $\tilde A(x)=B(x+\alpha) A(x) B(x)^{-1}$ satisfies
$\|\tilde A-R_\lambda\|_\epsilon \leq
\theta^{\frac {1} {2}+\delta-o(1)}=o(-\Im \lambda)$.
It follows that $\tilde A$
sends $\H$ into a precompact subset of $\H$.  So
$(\alpha,\tilde A)$ is uniformly
hyperbolic, which is impossible since it is conjugated to $(\alpha,A)$. 
\end{proof}

Let us note the following consequence:

\begin{lemma} \label {polyn}

Let $\alpha \in \R \setminus \Q$ be such that $\beta=0$
and let $A \in C^\omega_{\epsilon_0}(\R/\Z,\SL(2,\R))$
be such that $L(\alpha,A(x+t i))=0$ for $|t|<\epsilon_0$.
Let $0<\epsilon<\epsilon_0$ and let $\kappa_0=\min \{\frac {1} {2},1-\frac
{\epsilon} {\epsilon_0}\}$.  Then
\be
\|A_n\|_\epsilon \leq n^{\frac {1} {\kappa_0}-1+o(1)}.
\ee

\end{lemma}

\begin{proof}

Let $0<\kappa<\kappa_0$ and, take $\theta=n^{-1/\kappa}$ and use
Theorem \ref {comp} to get
$\|\tilde A_n\|=O(1)$ so that $\|A_n\|=O(\theta^{-1+\kappa})$.
\end{proof}

\section{Real conjugacies} \label {realsec}

Let $\theta>0$ be small and let $u=u_\theta$ and $s=s_\theta$ be the
unstable and stable directions of $(\alpha,R_{-i \theta} A)$, defined through
a band of size $\epsilon_0-o(1)$.

Take $0<\epsilon<\epsilon_0$ and let
$B$ be given by Lemma \ref {cho}, so that $B \cdot u=i$ and $B \cdot s=-i$. 
Then $\|B\|_{\epsilon_0-o(1)}^2 \leq \theta^{-\kappa(\theta)-o(1)}$ by
Theorem \ref {ang}.
Let $u'(z)=\overline {u(\overline {z})}$ and
$s'(z)=\overline {s(\overline {z})}$.

We have previously constructed
conjugacies by straightening up the pair $(u,s)$, but this was not
real-symmetric.  In order to do a real symmetric construction, we will make
use of the additional directions $u'$ and $s'$.

Let $U=B^{-1} \cdot \begin{pmatrix} i\\1 \end{pmatrix}$ and
$S=B^{-1} \cdot \begin{pmatrix} -i\\1 \end{pmatrix}$.
$U'(z)=\overline {U(\overline {z})}$ and
$S'(z)=\overline {S(\overline {z})}$.
Let $\Delta$ be the determinant of the matrix with columns $U$ and $U'$. 
Let also $\omega(z)=\|U(z)\|\|U'(z)\|$ so that $|\Delta(z)|/\omega(z)=
d(u(z),u'(z))$.  Note that
$\omega(z) \leq \theta^{-\kappa(\theta)-o(1)}$ over $|\Im
z|<\epsilon_0-o(1)$.

We note that at $\Im z=0$, $u$ and $u'$
must have distance at least of order
$\theta$, by Lemma \ref {field}.  This can be improved for many
$z$: by Lemma \ref {u}, $d(u(z),u'(z))>\theta^{1-\kappa(\theta)+o(1)}$ over
a set of probability $1-o(1)$.
(Recall that by Theorem \ref {comp}, we have $\kappa(\theta) \geq
1/2-o(1)$.)

For non real $z$ we have the following estimate:

\begin{lemma} \label {lat}

Let $z_*$ be such that $|\Im z_*|<\epsilon$.  Then
\begin{equation}
\sup_{\Im z=\Im z_*} |\Delta(z)| \leq
\theta+\max \{\theta,d(u(z_*),u'(z_*))\} \theta^{-o(1)}
\sup_{\Im z=\Im z_*} \omega(z).
\end{equation}

\end{lemma}

\begin{proof}

Notice that both $u$ and $u'$ are $\theta$-almost invariant under $A$, since
we have $R_{-i
\theta} A(z) \cdot u(z)=u(z+\alpha)$ and
$R_{i \theta} A(z) \cdot u'(z)=u'(z+\alpha)$.
Fix $\delta>0$ small.  By Lemma \ref {polyn},
\begin{equation} \label {bla5}
d(u(z_*+k \alpha),u'(z_*+k\alpha)) \leq
\theta^{-\delta} d(u(z_*),u'(z_*))+k \theta^{1-\delta}, \quad \mathrm {for}
\quad 0 \leq k \leq \theta^{-c \delta}
\end{equation}
for some $0<c<1$ depending on
$\epsilon$ and $\epsilon_0$.\footnote {We have
$d(A_k(z_*) \cdot u(z_*),
A_k(z_*) \cdot u'(z_*)) \leq C \|A_k(z_*)\|^2 d(u(z_*),u'(z_*))$.
By almost invariance
$d(A_{k-j+1}(z_*+(j-1) \alpha) \cdot u(z_*+(j-1) \alpha),
A_{k-j}(z_*+j \alpha) \cdot u(z_*+j \alpha)) \leq C \|A_{k-j}(z_*+j
\alpha)\|^2 \theta$ and
$d(A_{k-j+1}(z_*+(j-1) \alpha) \cdot u'(z_*+(j-1) \alpha),
A_{k-j}(z_*+j \alpha) \cdot u'(z_*+j \alpha)) \leq C \|A_{k-j}(z_*+j
\alpha)\|^2 \theta$.}
%\theta^{-\frac {\delta} {\frac {1} {\kappa}-1}}<\theta^{-\delta}$
%(here we fix some
%$0<\kappa<\min\{\frac {1} {2},1-\frac {\epsilon} {\epsilon_0}\}$).

We now use a Lagrange interpolation argument.
Let $\Delta_\theta(x)=\sum_{|k| \leq N_\theta} \hat \Delta_k e^{2 \pi i k
x}$ where $N_\theta
\geq -\ln \theta$
is chosen minimal so that $\|\Delta-\Delta_\theta\|_\epsilon<\theta^2$.
Since
$\|\Delta\|_{\epsilon_0-o(1)} \leq \theta^{-\kappa(\theta)-o(1)}$, we have
$N_\theta=O(-\ln \theta)$, so $2 N_\theta \leq \theta^{-c \delta}$.  In
particular (\ref {bla5}) implies 
\begin{equation} \label {bla7}
|\Delta(z_*+k \alpha)| \leq (\theta^{-2 \delta} d(u(z_*),u'(z_*))+\theta^{1-2
\delta}) \sup_{\Im z=\Im z_*} \omega(z).
\end{equation}

Let $q_n$ be maximal with
$q_n \leq 5 N_\theta$, and let $r \geq 1$ be minimal
such that $r q_n \geq 2N_\theta+2$.  In particular $rq_n<q_{n+1}$.  By
Theorem 6.1 of \cite {AJ2}, we have
\begin{equation}
\sup_{\Im z=\Im z_*} |\Delta_\theta(z)| \leq C q_{n+1}^{C r}
\sup_{0 \leq k \leq 2N_\theta} |\Delta_\theta(z_*+k \alpha)|.
\end{equation}
Since $\beta=0$, $C q_{n+1}^{C r} \leq C q_{n+1}^{C \frac {5 N_\theta}
{q_n}} \leq e^{o(N_\theta)} \leq \theta^{-o(1)}$.  It follows that
\begin{equation}
\sup_{\Im z=\Im z_*}
|\Delta_\theta(z)| \leq \theta^{-o(1)}
(\theta^2+\sup_{0 \leq k \leq 2N_\theta} |\Delta(z_*+k \alpha)|),
\end{equation}
so that
\begin{equation}
\sup_{\Im z=\Im z_*}
|\Delta(z)| \leq \theta^2+\theta^{-o(1)}
(\theta^2+\sup_{0 \leq k \leq 2N_\theta} |\Delta(z_*+k \alpha)|)
\end{equation}
which together with (\ref {bla7}) implies the desired estimate.
\end{proof}

\begin{lemma} \label {u1}

For $0<\epsilon'<\epsilon$, if $\inf_{|\Im z|<\epsilon}
d(u(z),u'(z))<\theta^l$ with $0 \leq l \leq 1$
then $\|U\|^2_0 \leq \max \{\theta^{\kappa(\theta)-o(1)},
\theta^{-(1-\frac
{\epsilon} {\epsilon'}) \kappa(\theta)-\frac {\epsilon}{\epsilon'} (1-l)
-o(1)}\}$.  In particular, if $\|U\|_0^2 \geq 2$ we have
$l \leq 1-(1-\frac {\epsilon'} {\epsilon}) \kappa(\theta)+o(1)$.

\end{lemma}

\begin{proof}

Define $\gamma_t$, $0 \leq t \leq \epsilon$ so that
$\theta^{-\gamma_t}=\sup_{|\Im z| \leq t}
\omega(z)$.  In particular $\|U\|_0^2=\theta^{-\gamma_0}$.
We have
%$\gamma_0 \leq \gamma_{\epsilon'} \leq \gamma_\epsilon$,
$\gamma_\epsilon \leq \kappa(\theta)+o(1)$ and
subharmonicity gives $\gamma_{\epsilon'} \leq
\frac {\epsilon'} {\epsilon} \kappa(\theta)+(1-\frac {\epsilon'}
{\epsilon})\gamma_0+o(1)$.
Noting that $i \Delta$ is real-symmetric, Lemma \ref {lat} gives
$\|\Delta\|_0 \leq \theta+\theta^{l-\gamma_{\epsilon'}-o(1)}$.
By Lemma \ref {u}, we conclude that
\begin{equation} \label {38}
\omega(z) \leq
\theta^{-1+\kappa(\theta)-o(1)} |\Delta(z)| \leq
\theta^{\kappa(\theta)-o(1)}+
\theta^{l-\gamma_{\epsilon'}-1+\kappa(\theta)-o(1)}
\end{equation}
for $\Im z=0$ with probability $1-o(1)$.
Since there is also an upper bound $\omega(z) \leq
\theta^{-\kappa(\theta)-o(1)}$
through $|\Im z|<\epsilon$,
this implies that (\ref {38}) in fact holds for all $\Im z=0$
(subharmonicity).  Thus $\gamma_0 \leq \max \{-\kappa(\theta),
\gamma_{\epsilon'}+1-l-\kappa(\theta)\}+o(1)$.  In particular, if $\gamma_0
\geq -\kappa(\theta)+o(1)$ we have
$\gamma_0 \leq
\frac {\epsilon'} {\epsilon} \kappa(\theta)+(1-\frac {\epsilon'}
{\epsilon})\gamma_0+1-l-\kappa(\theta)+o(1)$, so that $\gamma_0 \leq
(1-\frac{\epsilon}{\epsilon'})\kappa(\theta)+\frac {\epsilon} {\epsilon'}
(1-l)+o(1)$.
\end{proof}

\begin{thm} \label {real}

Let $\alpha \in \R \setminus \Q$ be such $\beta=0$ and let $A \in
C^\omega_{\epsilon_0}(\R/\Z,\SL(2,\R))$ be such that $L(\alpha,A(x+t i))=0$
for $|t|<\epsilon_0$. Let $0<\epsilon'<\epsilon_0$ and let
$\kappa'_0=\frac {1}
{2} (1-\frac {\epsilon'} {\epsilon_0})$.  Then for every $\theta>0$
sufficiently small,
there exists $B_r \in
C^\omega_{\epsilon'}(\R/\Z,\SL(2,\R))$ such that
$\|B_r\|^2_{\epsilon'} \leq \theta^{-1+\kappa'_0-o(1)}$ and $\tilde
A_r(x)=B_r(x+\alpha)A(x)B_r(x)^{-1}$ satisfies $\|\tilde
A_r-R_\lambda\|_{\epsilon'} \leq \theta^{\kappa'_0-o(1)}$ for some $\lambda
\in \R$.

\end{thm}

\begin{proof}

Fix $0<\kappa'<\kappa'_0$, take $\epsilon''$ slightly larger
than $\epsilon'$ and $\kappa''$ slightly larger than $\kappa'$.  Take
$0<\epsilon<\epsilon_0$ close to $\epsilon_0$.
Assume first that $\|U\|_0^2 \geq 2$.
Since $\kappa(\theta) \geq \frac {1} {2}-o(1)$,
Lemma \ref {u1} implies that
$d(u(z),u'(z)) \geq \theta^{1-\kappa''}$ for $|\Im z|<\epsilon''$.
%(using a suharmonic proxy for the logarithm of the inverse of the angle).

We can then conclude following the argument of the proof of Theorem \ref
{comp}.  Consider $B'_r \in C^\omega_{\epsilon''}(\R/\Z,\SL(2,\R))$
such that $B'_r \cdot u=i$ and $B'_r \cdot u'=-i$.  This matrix will satisfy
the bound $\|B_r'\|_{\epsilon''}^2=O(\theta^{-1+\kappa''})$.
Since $u$ and $u'$ are
$\theta$-almost invariant under $A$, $A'(z)=B_r'(z+\alpha) A(z) B_r(z)^{-1}$
takes $i$ and $-i$ somewhere $\theta^{\kappa''}$ close to themselves.  So
$A'(z)=A''(z)R_{\psi(z)}$ where $A''$ and $\psi$ (which takes values in
$\C/\Z$) are real-symmetric and
$\|A''-\id\|=O(\theta^{\kappa''})$.  Write $\psi(x)=kx+\phi(x)$ with $k \in
\Z$ and $\phi$ taking values in $\C$.  Solve the cohomological equation
$\phi(x)=w(x+\alpha)-w(x)+\lambda$ with $w$ real-symmetric and bounded.
Then $R_{w(z+\alpha)} A''(z)
R_{\psi(z)} R_{-w(z)}=R_{w(z+\alpha)} A''(z) R_{-w(z+\alpha)}
R_{k z+\lambda}$.  Since the Lyapunov exponent vanishes through
the band, $k=0$.
The result then follows by taking $B_r(z)=R_{w(z)} B'_r(z)$.

Assume now that $\|U\|^2_0<2$.  Then $\|S\|^2_0 \geq 2$ (since $B$ has
determinant $1$), so we can do the whole argument with $s$ and $s'$ instead
of $u$ and $u'$ to conclude.
\end{proof}

%\appendix
%\input{matrices}


\begin{thebibliography}{AFK}

%\bibitem[A1]{A1} Avila, A. The absolutely continuous spectrum of the almost
%Mathieu operator.  Preprint (www.arXiv.org).

\bibitem[Am]{amor} Amor, S.
H\"older continuity of the rotation number for quasiperiodic cocycles in
$\SL(2,\R)$.  Comm. Math. Phys. 287 (2009), 565-583.

\bibitem[A1]{A1} Avila, A. Global theory of one-frequency Schr\"odinger
operators.  Acta Math. 215 (2015), 1-54.

%\bibitem[A1]{A1} Avila, A. Global theory of one-frequency Schr\"odinger
%operators I: stratified analyticity of the Lyapunov exponent and the
%boundary of nonuniform hyperbolicity.  Preprint (www.arXiv.org).

%\bibitem[A2]{A2} Avila, A. Global theory of one-frequency Schr\"odinger
%operators II:
%acriticality and the finiteness of phase transitions for typical potentials.
%Preprint (www.arXiv.org).

\bibitem[A2]{arac} Avila, A. Almost reducibility and absolute continuity I.
Preprint (www.impa.br/$\sim$avila/).

\bibitem[A3]{density} Avila, A. Density of positive Lyapunov exponents for
$\SL(2,\R)$ cocycles.  Journal of the American Mathematical Society 24
(2011), 999-1014.

%\bibitem[AD]{AD} Avila, Artur; Damanik, David Generic singular spectrum for
%ergodic Schr\"odinger operators.  Duke Math. J.  130  (2005),  no. 2,
%393--400.

\bibitem[AFK]{AFK} Avila, A.; Fayad, B.; Krikorian, R. A KAM scheme for
$\SL(2,\R)$ cocycles with Liouvillean frequencies.
Geometric and Functional
Analysis 21 (2011), 1001-1019.

\bibitem[AJ]{AJ2} Avila, A.; Jitomirskaya, S. Almost localization and almost
reducibility.
Journal of the European Mathematical Society 12 (2010),
93-131.

\bibitem[AK1]{AK1} Avila, A.; Krikorian, R.
Reducibility or non-uniform hyperbolicity for
quasiperiodic Schr\"odinger cocycles.  Ann. of Math. 164 (2006), 911-940.

\bibitem[AK2]{AK2} Avila, A.; Krikorian, R.
Monotonic cocycles.  Invent. Math. 202 (2015), 271-331.

\bibitem[AYZ]{AYZ} Avila, A.; You, J.; Zhou, Q.
Dry Ten Martini problem in the non-critical case.
Preprint (www.arxiv.org).

%\bibitem[AS]{AS} Avron, Joseph; Simon, Barry Almost periodic Schr\"odinger
%operators. II. The integrated density of states.  Duke Math. J.  50  (1983),
%no. 1, 369--391.

%\bibitem[Bj1]{Bj} Bjerkl\"ov, K. Explicit examples of arbitrarily large
%analytic ergodic
%potentials with zero Lyapunov exponent.  Geom. Funct. Anal.  16  (2006),
%no. 6, 1183-1200.

%\bibitem[Bj2]{Bj1} Bjerkl\"ov, Kristian Positive Lyapunov exponent and
%minimality for a class of one-dimensional quasi-periodic Schr\"odinger
%equations.  Ergodic Theory Dynam. Systems  25  (2005),  no. 4, 1015--1045.

%\bibitem[B]{B} Bourgain, J. Green's function estimates for lattice
%Schr\"odinger operators and applications. Annals of Mathematics Studies, 158.
%Princeton University Press, Princeton, NJ, 2005. x+173 pp.

\bibitem[BG]{BG} Bourgain, J.; Goldstein, M.
On nonperturbative localization with quasi-periodic potential.  Ann. of
Math. (2)  152  (2000),  no. 3, 835--879.

\bibitem[BJ1]{BJ1} Bourgain, J.; Jitomirskaya, S.
Continuity of the Lyapunov exponent for quasiperiodic operators with
analytic potential.
J. Statist. Phys.  108  (2002),  no. 5-6,
1203--1218.

\bibitem[BJ2]{BJ2} Bourgain, J.; Jitomirskaya, S.
Absolutely continuous spectrum for 1D quasiperiodic operators.  Invent.
Math.  148  (2002),  no. 3, 453--463.

\bibitem[DS]{DS} Dinaburg, E.; Sinai, Y.
The one-dimensional Schr\"odinger equation with a quasiperiodic potential.
Funkts. Anal. Prilozh. 9 (1975), 8-21.

\bibitem[E]{E} Eliasson, L. H. Floquet solutions for the $1$-dimensional
quasi-periodic Schr\"odinger equation.  Comm. Math. Phys.  146  (1992),  no.
3, 447--482.

\bibitem[GS1]{GS} Goldstein, M.; Schlag, W.
H\"older continuity of the integrated
density of states for quasi-periodic Schr\"odinger equations and averages of
shifts of subharmonic functions.  Ann. of Math. (2)  154  (2001),  no. 1,
155--203.

\bibitem[GS2]{GS1} Goldstein, M.; Schlag, W.
Fine properties of the integrated density of states and a quantitative
separation property of the Dirichlet eigenvalues.  Geom. Funct. Anal.  18
(2008),  no. 3, 755--869.

\bibitem[GS3]{GS2} Goldstein, M.; Schlag, W.
On resonances and the formation of gaps in the spectrum of quasi-periodic
Schr\"odinger equations.  Ann. of Math. 173 (2011), 337-475.

\bibitem[H1]{H1} Herman, M.
Sur la conjugaison diff\'erentiables des diff\'eomorphismes du cercle \`a
des rotations.  Pub. Math. IHES 49 (1979), 5-233.

\bibitem[H2]{H} Herman, M.
Une m\'ethode pour minorer les exposants de Lyapounov et quelques exemples
montrant le caract\`ere local d'un th\'eor\`eme d'Arnol'd et de Moser sur le tore
de dimension $2$.  Comment. Math. Helv. 58 (1983), no. 3, 453--502. 

%\bibitem[HPS]{HPS} Hirsch, M. W.; Pugh, C. C.; Shub, M. Invariant manifolds.
%Lecture Notes in Mathematics, Vol. 583. Springer-Verlag, Berlin-New York,
%1977. ii+149 pp.

%\bibitem[J]{J} Jitomirskaya, S.
%Metal-insulator transition for the almost Mathieu operator.  Ann. of Math.
%(2)  150  (1999),  no. 3, 1159--1175.

%\bibitem[JKS]{JKS} Jitomirskaya S.; Koslover D.; Schulteis M.
%Continuity of the Lyapunov exponent for quasiperiodic Jacobi matrices.
%Preprint, 2004, to appear in Ergodic Theory and Dynamical Systems.

%\bibitem[JL]{JL2} Jitomirskaya, S.; Last, Y. Power law subordinacy and
%singular spectra. II. Line operators.  Comm. Math. Phys.  211  (2000),  no.
%3, 643--658.

\bibitem[SS]{SS} Sorets, E.; Spencer, T.
Positive Lyapunov exponents for Schr\"odinger operators with quasi-periodic
potentials.
Comm. Math. Phys. 142 (1991), no. 3, 543--566. 

\bibitem[Y]{Y} Yoccoz, J.-C.
Conjugaison diff\'erentiable des diff\'eomorphismes du cercle dont le nombre
de rotation v\'erifie une condition diophantienne.
Ann. Sci. ENS 17 (1984), 333-359.

\end{thebibliography}
\end{document}